\documentclass[11pt]{article}
\usepackage{amsmath,amsthm,amsfonts,amssymb,enumerate,calc,graphicx,iwona}
\usepackage[colorlinks=true,citecolor=black,linkcolor=black,urlcolor=blue]{hyperref}
\usepackage[capitalise]{cleveref}
\usepackage[lmargin=33mm,rmargin=33mm,bmargin=29mm,tmargin=29mm]{geometry}
\usepackage[numbers,sort&compress]{natbib}
\renewcommand{\baselinestretch}{1.125}
\usepackage[hang]{footmisc} 

\renewcommand{\thefootnote}{\fnsymbol{footnote}}	
\allowdisplaybreaks
\newcommand\DateFootnote{
\begingroup
\renewcommand\thefootnote{}
\footnote{15th November, 2015, revised 15th July 2016}
\setcounter{footnote}{0}
\vspace*{-3ex}
\endgroup}

\makeatletter
\renewcommand\section{\@startsection {section}{1}{\z@}%
                                   {-3ex \@plus -1ex \@minus -.2ex}%
                                   {2ex \@plus.2ex}%
                                   {\normalfont\large\bfseries}}
\renewcommand\subsection{\@startsection{subsection}{2}{\z@}%
                                     {-2.5ex\@plus -1ex \@minus -.2ex}%
                                     {1.5ex \@plus .2ex}%
                                     {\normalfont\normalsize\bfseries}}
\renewcommand\subsubsection{\@startsection{subsubsection}{3}{\z@}%
                                     {-2ex\@plus -1ex \@minus -.2ex}%
                                     {1ex \@plus .2ex}%
                                     {\normalfont\normalsize\bfseries}}
 \renewcommand\paragraph{\@startsection{paragraph}{4}{\z@}%
                                    {1.5ex \@plus.5ex \@minus.2ex}%
                                    {-1em}%
                                    {\normalfont\normalsize\bfseries}}
\renewcommand\subparagraph{\@startsection{subparagraph}{5}{\parindent}%
                                       {1.5ex \@plus.5ex \@minus .2ex}%
                                       {-1em}%
                                      {\normalfont\normalsize\bfseries}}
\makeatother

\newcommand{\Zbl}[1]{Zbl:\,\href{http://www.zentralblatt-math.org/zmath/en/search/?q=an:#1}{#1}}
\newcommand{\arXiv}[1]{arXiv:\,\href{http://arxiv.org/abs/#1}{#1}}
\newcommand{\msn}[1]{MR:\,\href{http://www.ams.org/mathscinet-getitem?mr=MR#1}{#1}}
\newcommand{\MSN}[2]{MR:\,\href{http://www.ams.org/mathscinet-getitem?mr=MR#1}{#1}}
\newcommand{\doi}[1]{doi:\,\href{http://dx.doi.org/#1}{#1}}

\theoremstyle{plain}
\newtheorem{thm}{Theorem}
\newtheorem{lem}[thm]{Lemma}

\newtheorem{prop}[thm]{Proposition}
\newtheorem{conj}[thm]{Conjecture}

\newcommand{\floor}[1]{\lfloor{#1}\rfloor}
\newcommand{\cliques}{\textup{cliques}}
\newcommand{\half}{\tfrac12}

\begin{document}

\vspace*{2ex}
{\Large\bfseries\boldmath\scshape Cliques in Graphs Excluding a Complete Graph Minor}

\DateFootnote

\medskip
\bigskip
{\large 
David~R.~Wood\,\footnotemark[1]
}

\bigskip
\bigskip
\emph{Abstract.} This paper considers the following question: What is the maximum number of $k$-cliques in an $n$-vertex graph with no $K_t$-minor? This question generalises the extremal function for $K_t$-minors, which corresponds to the $k=2$ case. The exact answer is given for $t\leq 9$ and all values of $k$. We also determine the maximum total number of cliques in an $n$-vertex graph with no $K_t$-minor for $t\leq 9$. Several observations are made about the case of general $t$. 

\footnotetext[1]{School of Mathematical Sciences, Monash University, Melbourne, Australia
  (\texttt{david.wood@monash.edu}).\\ Research supported by  the Australian Research Council.}

\renewcommand{\thefootnote}{\arabic{footnote}}

\section{Introduction}
\label{Intro}

A basic question of extremal graph theory asks: for a class $\mathcal{G}$ of graphs, what is the maximum number of edges in an $n$-vertex graph in $\mathcal{G}$? The answer is called the extremal function for $\mathcal{G}$. Consider the following two classical examples. Tur\'an's Theorem \citep{Turan41} says that every $n$-vertex graph with no $K_t$-subgraph has at most $(\frac{t-2}{2t-2})n^2$ edges, with equality only for the complete $(t-1)$-partite graph with $\frac{n}{t-1}$ vertices in each colours class (called the \emph{Tur\'an graph}). And Euler's formula implies that the maximum number of edges in a planar graph with $n\geq 3$ vertices equals $3n-6$. 

One way to generalise these results is to consider cliques instead of edges. A \emph{clique} in a graph is a set of pairwise adjacent vertices. A \emph{$k$-clique} is a clique of cardinality $k$. Since a 2-clique is simply an edge, the following natural generalisations of the above question arise: For a class $\mathcal{G}$ of graphs, 
\begin{itemize}
\item what is the maximum number of $k$-cliques in an $n$-vertex graph in $\mathcal{G}$, and 
\item what is the maximum number of cliques in an $n$-vertex graph in $\mathcal{G}$? 
\end{itemize}

Let $\cliques(G)$ be the number of cliques in a graph $G$.  Let $\cliques(G,k)$ be the number of $k$-cliques in a graph $G$. Of course, $\cliques(G,0)=1$, $\cliques(G,1)=|V(G)|$ and $\cliques(G,2)=|E(G)|$. 

\citet{Zykov49} generalised Tur\'an's Theorem by answering the above questions for the class of graphs with no $K_t$-subgraph. He proved that for $t>k\geq 0$, every graph with $n\geq k$ vertices and no $K_t$-subgraph contains at most $\binom{t-1}{k}(\frac{n}{t-1})^{k}$ cliques of size $k$, and every graph with $n$ vertices and no $K_t$-subgraph contains at most $(\frac{n}{t-1}+1)^{t-1}$ cliques. Both bounds are tight for the Tur\'an graph. Bounds on the number of $k$-cliques in graphs of given maximum degree have been extensively studied \citep{CR14,EG14,Galvin11,GLS15,ACM12,Wood-GC07}. Several papers have established upper bounds on the number of $k$-cliques in terms of the number of vertices and the number of edges, or more generally, in terms of the number of $(\leq k-1)$-cliques \citep{CR11,Frohmader10,Eckhoff-DM04,Eckhoff-DM99,Rivin02}.

For planar graphs, \citet{HS79} proved that  the maximum number of triangles is $3n-8$, and \citet{Wood-GC07} proved that the maximum number of 4-cliques is $n-3$, and in total the maximum number of cliques is $8n-16$. See \citep{PY-IPL81} for earlier upper bounds for planar graphs and see \citep{DFJSW} for an extension to arbitrary surfaces. 

This paper considers these questions in graph classes defined by an excluded minor, thus generalising the above results for planar graphs. This direction has been recently pursued by several authors \citep{ReedWood-TALG,NSTW-JCTB06,FOT,LO15,FW16}. These works have focused on asymptotic results when the excluded minor is a general complete graph $K_t$. The primary focus of this paper is exact results, when the excluded minor is $K_3$, $K_4$, $K_5$, $K_6$, $K_7$, $K_8$ or $K_9$ (Sections~\ref{sectK7}--\ref{sectK9}). We also make several observations and conjectures about  general $K_t$-minor-free graphs (Sections~\ref{KtTotal}--\ref{KtSpecific}). 

While bounds on the number of cliques in minor-closed classes are of  independent interest, such results have had diverse applications, including the asymptotic enumeration of minor-closed classes \citep{NSTW-JCTB06}, and in the analysis of an algorithm for finding small separators \citep{ReedWood-TALG}, which in turn has been applied to finding shortest paths \citep{MT-DAM09} and in matrix sparsification \citep{AY13} for example. 

Let $\cliques(n,t,k)$ be the maximum number of $k$-cliques in a $K_t$-minor-free graph on $n$ vertices. 
Let $\cliques(n,t)$ be the maximum number of cliques in a $K_t$-minor-free graph on $n$ vertices. 
Of course, if $n\leq t-1$ then $K_n$ is $K_t$-minor-free, in which case 
\begin{equation}
\label{smalln}
\cliques(n,t,k)=\binom{n}{k}\quad \text{and}\quad\cliques(n,t)=2^n.
\end{equation}



The following example provides an important lower bound on $\cliques(n,t,k)$. For an integer $\ell\geq 1$, an \emph{$\ell$-tree} is a graph defined recursively as follows. First, the complete graph $K_{\ell}$ is an $\ell$-tree. Then, if $C$ is an $\ell$-clique in an $\ell$-tree, then the graph obtained by adding a new vertex adjacent only to $C$, is also a an $\ell$-tree. Every $\ell$-tree has tree-width at most $\ell$, and thus contains no $K_{\ell+2}$-minor.  Observe that for every $\ell$-tree $G$ with $n$ vertices, $\cliques(G,k)=\binom{\ell}{k-1}(n-\frac{(\ell+1)(k-1)}{k})$ and $\cliques(G) = 2^{\ell}(n-\ell+1)$. Hence for $n\geq t-2$ and $t>k\geq 1$, 
\begin{align}
\label{LowerBound}
\cliques(n,t,k) & \geq \binom{t-2}{k-1}\left(n-\frac{(k-1)(t-1)}{k}\right)\\
\label{TotalLowerBound}
\cliques(n,t) & \geq 2^{t-2}(n-t+3) .
\end{align}

The results of this paper show that these lower bounds hold with equality for many values of $t$ and $k$. This is the case for  $n\in\{t-2,t-1\}$ by \eqref{smalln}.
%
%

The starting point for our investigation is the following classical result by \citet{Dirac64} for $t\in\{3,4,5\}$ and by \citet{Mader68} for $t\in\{6,7\}$.

\begin{thm}[\citep{Dirac64,Mader68}]
\label{Mader}
For $t\in\{3,4,5,6,7\}$, the maximum number of edges in a $K_t$-minor-free graph on $n\geq t-2$ vertices satisfies
$$\cliques(n,t,2)=(t-2)\left(n-\frac{t-1}{2}\right).$$
\end{thm}
  
\citet{Mader68}  observed that \cref{Mader} does not hold with $t=8$. Let $K_{c\times 2}$ be the complete $c$-partite graph $K_{2,2,\dots,2}$ with $n=2c$ vertices, which can be thought of as $K_{2c}$ minus a perfect matching. In the $t=8$ case, \cref{Mader} would give a bound of $6n-21$ on the number of edges in a $K_8$-minor-free graph, whereas  \citet{Mader68} observed that $K_{2,2,2,2,2}$ has $n=10$ vertices, $40>6\cdot10-21$ edges, and contains no $K_8$-minor.  $K_{c\times 2}$ will be an important example throughout this paper. In general, \citet{Wood-GC07} proved that $K_{c\times 2}$ contains no $K_t$-minor where $t=\floor{\frac{3}{2}c}+1$, and a $K_{t-1}$-minor in $K_{c\times 2}$ is obtained from a $K_c$  subgraph by contracting a $\floor{\frac{c}{2}}$-edge matching in the remaining graph. 

  
The following theorems summarise the main contributions of this paper.

\begin{thm}
\label{Summary}
For $t\in\{3,4,\dots,9\}$ and $k\in\{1,2,\dots,t-1\}$ and $n\geq t-2$, 
$$\cliques(n,t,k)=\binom{t-2}{k-1}\left(n-\frac{(k-1)(t-1)}{k}\right),$$
except for $(t,k)\in\{(8,2),(9,2),(9,3)\}$ and certain values of $n$ (made precise below) in which case
$$\cliques(n,t,k)=\binom{t-2}{k-1}\left(n-\frac{(k-1)(t-1)}{k}\right)+1.$$
\end{thm}

The case $k=2$ (that is, number of edges) was established by \citet{Dirac64} for $t\in\{3,4,5\}$, by \citet{Mader68} for $t\in\{6,7\}$, by \citet{Jorgensen94} for $t=8$, and by \citet{ST06} for $t=9$. Note that our proof depends on the case $k=2$ and does not reprove the existing results. 

For the total number of cliques, we prove:

\begin{thm}
\label{SummaryTotal}
For $t\in\{3,4,5,6,7,8,9\}$, the maximum number of cliques in a $K_t$-minor-free graph on $n\geq t-2$ vertices satisfies $$\cliques(n,t)=2^{t-2}(n-t+3).$$ 
\end{thm}

We employ the following notation. For a vertex $v$ in a graph $G$, let $N(v)$ be the set of neighbours of $v$, let $N[v]:=N(v)\cup\{v\}$, let $G(v):=G[N(v)]$, and let $G[v]:=G[N[v]]$.

\section{Cockades}
\label{Cockades}

This section introduces a well known construction that will be important later. Let $H$ be a graph containing a $k$-clique. An $(H,k)$-cockade is defined recursively as follows. First, $H$ is an \emph{$(H,k)$-cockade}. And if $H_1$ and $H_2$ are $(H,k)$-cockades, then the graph obtained from pasting $H_1$ and $H_2$ on a $k$-clique is an  \emph{$(H,k)$-cockade}. It is easy to count cliques in cockades. 

\begin{lem}
\label{Pasting}
For $i\in\{1,2\}$, let $G_i$ be a graph with $n_i$ vertices such that $\cliques(G_i,k)=a(n_i-r)+\binom{r}{k}$, for some fixed $a$ and $r\geq k$. Let $G$ be obtained by pasting  $G_1$ and $G_2$ on an $r$-clique. Then $G$ has $n=n_1+n_2-r$ vertices, and $\cliques(G,k)=a(n-r)+\binom{r}{k}$. 
\end{lem}

\begin{proof}
\begin{align*}
\cliques(G,k)
&=\cliques(G_1,k)+\cliques(G_2,k)-\binom{r}{k}\\
&=a(n_1-r)+\binom{r}{k}+a(n_2-r)+\binom{r}{k}-\binom{r}{k}\\
&=a(n_1+n_2-2r)+\binom{r}{k}\\
&=a(n-r)+\binom{r}{k}.\qedhere
\end{align*}
\end{proof}

\begin{lem}
\label{TotalPasting}
For $i\in\{1,2\}$, let $G_i$ be a graph with $n_i$ vertices such that $\cliques(G_i)=a(n_i-r)+2^r$, for some fixed $a$ and $r\geq 0$. Let $G$ be obtained by pasting  $G_1$ and $G_2$ on an $r$-clique. Then $G$ has $n=n_1+n_2-r$ vertices, and $\cliques(G)=a(n-r)+2^r$. 
\end{lem}

\begin{proof}
\begin{align*}
\cliques(G)
&=\cliques(G_1)+\cliques(G_2)-2^r\\
&=(a(n_1-r)+2^r)+(a(n_2-r)+2^r)-2^r\\
&=a(n_1+n_2-2r)+2^r\\
&=a(n-r)+2^r.\qedhere
\end{align*} 
\end{proof}
%
%

\begin{lem}
\label{Cockade}
For every $(K_{c\times 2},c)$-cockade $G$ on $n$ vertices and for $k\in\{0,1,\dots,c\}$,
\begin{align*}
\cliques(G,k)&=\frac{1}{c}\binom{c}{k}(2^k-1)(n-c)+\binom{c}{k}\text{, and}\\
\cliques(G) & = \frac{1}{c}(3^c-2^c)(n-c)+2^c\enspace.
\end{align*}
\end{lem}

\begin{proof}
The first claim follows from \cref{Pasting} with $r=c$ and $a=\frac{1}{c}\binom{c}{k}(2^k-1)$ since $\cliques(K_{c\times 2},k)=\binom{c}{k}2^k=a(2c-c)+\binom{c}{k}$. The second claim follows from \cref{TotalPasting} with $r=c$ and $a=\frac{1}{c}(3^c-2^c)$ since  $\cliques(K_{c\times 2})=3^c=a(2c-c)+2^c$.
\end{proof}

\section{\boldmath $K_t$-minor-free graphs with $t\leq 7$}
\label{sectK7}

This section proves Theorems~\ref{Summary} and \ref{SummaryTotal} for $t\leq 7$. 

\begin{thm}
\label{MaderGen}
For $t\in\{3,4,5,6,7\}$ and $k\in\{1,2,\dots,t-1\}$ and $n\geq t-2$, 
the maximum number of $k$-cliques in a $K_t$-minor-free graph on $n$ vertices satisfies
$$\cliques(n,t,k)=\binom{t-2}{k-1}\left(n-\frac{(k-1)(t-1)}{k}\right).$$
\end{thm}

\begin{proof}
The lower bound is provided by \eqref{LowerBound}. For the upper bound, we proceed by induction on $n+k$. The claim is trivial if $k=1$. \cref{Mader} proves the claim when $k=2$. Now assume that $k\geq 3$. In the case $n=t-2$ the claimed upper bound on $\cliques(n,t,k)$ is $\binom{n}{k}$, which obviously holds. Let $G$ be a $K_t$-minor-free graph on $n\geq t-1$ vertices. 

First suppose that $\deg(v)\leq t-2$ for some vertex $v$ of $G$. Each clique of $G$ either contains $v$ or does not contain $v$. The $k$-cliques of $G$ that contain $v$ are in 1--1 correspondence with the $(k-1)$-cliques of $G(v)$. And the cliques of $G$ that do not contain $v$ are exactly the cliques of $G-v$. Thus $$\cliques(G,k)=\cliques(G(v),k-1)+\cliques(G-v,k),$$ which by induction is at most 
$$\binom{\deg(v)}{k-1}+ 
\binom{t-2}{k-1}\left(n-1-\frac{(k-1)(t-1)}{k}\right)
\leq \binom{t-2}{k-1}\left(n-\frac{(k-1)(t-1)}{k}\right),$$ 
as desired. 

Now assume that $G$ has minimum degree at least $t-1$. For each $k$-clique $C$ in $G$ send a charge of $\frac{1}{k}$ to each vertex in $C$. The charge received by each vertex $v$ equals $\frac{1}{k}\cliques(G(v),k-1)$. Since the total charge equals $\cliques(G,k)$, and $G(v)$ is $K_{t-1}$-minor-free, 
\begin{align*}
\cliques(G,k)
=\frac{1}{k}\sum_{v\in V(G)}\cliques(G(v),k-1)
\leq
\frac{1}{k}\sum_{v\in V(G)}\cliques(\deg(v),t-1,k-1).
\end{align*}
(This argument is essentially that of \citet[Lemma~5]{FOT}.)\ 
Since $\deg(v)\geq t-1$, by induction, 
\begin{align}
\cliques(G,k)
&\leq\frac{1}{k}\sum_{v\in V(G)}\binom{t-3}{k-2}\left(\deg(v)-\frac{(k-2)(t-2)}{k-1}\right)\nonumber\\
&=\frac{1}{k}\binom{t-3}{k-2}\left( 2|E(G)|-\frac{(k-2)(t-2)n}{k-1}\right)\label{CliquesEdges}\\
&\leq \frac{1}{k}\binom{t-3}{k-2}\left(2\,\cliques(n,t,2)-\frac{(k-2)(t-2)n}{k-1}\right).\nonumber
\end{align}
By Theorem~\ref{Mader}, 
\begin{align*}
\cliques(G,k)
&\leq\frac{1}{k}\binom{t-3}{k-2}\left(2(t-2)\left(n-\frac{t-1}{2}\right)-\frac{(k-2)(t-2)n}{k-1}\right)\\
&=\frac{(t-2)}{k}\binom{t-3}{k-2}\left(2n-(t-1)-\frac{(k-2)n}{k-1}\right)\\
&=\frac{(t-2)}{k}\binom{t-3}{k-2}\left(\left(\frac{k}{k-1}\right)n-(t-1)\right)\\
&=\frac{(t-2)}{(k-1)}\binom{t-3}{k-2}\big(n-\frac{(k-1)(t-1)}{k}\big)\\
&=\binom{t-2}{k-1}\left(n-\frac{(k-1)(t-1)}{k}\right),
\end{align*}
as desired.
\end{proof}

\begin{thm}
\label{Total7}
For $t\in\{3,4,5,6,7\}$, the maximum number of cliques in a $K_t$-minor-free graph on $n\geq t-2$ vertices equals $2^{t-2}(n-t+3)$. 
\end{thm}

\begin{proof}
The lower bound is \eqref{TotalLowerBound}. The upper bound follows from Theorem~\ref{MaderGen} since 
\begin{equation*}
1+\sum_{k=1}^{t-1}\binom{t-2}{k-1}\left(n-\frac{(k-1)(t-1)}{k}\right) =2^{t-2}(n-t+3).\qedhere
\end{equation*}
\end{proof}

Theorem~\ref{MaderGen} and \ref{Total7} were previously proved for $t=5$ by \citet{Wood-GC07} (using a different method). 

\section{\boldmath $K_8$-minor-free graphs}
\label{sectK8}

Determining the maximum number of cliques in a $K_8$-minor-free graph is more difficult than in the $t\leq 7$ cases, since $K_{2,2,2,2,2}$ would be a counterexample to \cref{Mader} with $t=8$ (see \cref{Intro}). In fact, every $(K_{2,2,2,2,2},5)$-cockade would be a counterexample. \citet{Jorgensen94} showed these are the only counterexamples. 

\begin{thm}[\citep{Jorgensen94}]
\label{Jorgensen}
Every $K_8$-minor-free graph on $n\geq6$ vertices has at most $6n-21$ edges or is a $(K_{2,2,2,2,2},5)$-cockade (which has $6n-20$ edges). 
\end{thm}

We now prove \cref{Summary} for $K_8$-minor-free graphs with $k\geq3$. 

\begin{thm}
\label{K8}
For $k\in\{3,4,5,6,7\}$ the maximum number of $k$-cliques in a $K_8$-minor-free graph on $n\geq 6$ vertices satisfies
$$\cliques(n,8,k)=\binom{6}{k-1}\left(n-\frac{7(k-1)}{k}\right).$$
\end{thm}

\begin{proof}
The lower bound is provided by \eqref{LowerBound}. For the upper bound, let $G$ be a $K_8$-minor-free graph on $n\geq 7$ vertices. We proceed by induction on $n$ with $k$ fixed. In the base case with $n\in\{6,7\}$, the result holds by \eqref{smalln}. Now assume that $n\geq 8$. 

Say $G$ is a $(K_{2,2,2,2,2},5)$-cockade. Then $\cliques(G,6)=\cliques(G,7)=0$. By \cref{Cockade}, $\cliques(G,3) = 14n-60 \leq 15n-70$ (since $n\geq 10$) and $\cliques(G,4)= 15n-70 \leq 20n-105$ and $\cliques(G,5) = \frac{31}{5}n-30 \leq 15n-84$. This show that $\cliques(G,k) \leq \binom{6}{k-1}n-\binom{7}{k}(k-1)$. Now assume that $G$ is not a $(K_{2,2,2,2,2},5)$-cockade. Thus $|E(G)|\leq 6n-21$ by Theorem~\ref{Jorgensen}. 

The remainder of the proof is analogous to the proof of Theorem~\ref{MaderGen}, so we sketch it briefly. First, delete a vertex of degree at most 5 and apply induction. Now assume minimum degree at least 6. Charge each $k$-clique to its vertices, and count the charge at each vertex $v$ with respect to $\deg(v)$ and the number of $(k-1)$-cliques in $G(v)$, which is $K_7$-minor-free (applying Theorem~\ref{MaderGen}). 
Counting the total charge, \eqref{CliquesEdges} gives
$$\cliques(G,k) \leq \frac{1}{k}\binom{5}{k-2}\left( 2|E(G)|-\frac{(k-2)3n}{k-1}\right).$$
Since $|E(G)|\leq 6n-21$, it follows by the same analysis used in the proof of \cref{MaderGen} 
that  $\cliques(G,k)\leq\binom{6}{k-1}(n-\frac{7(k-1)}{k})$. 
\end{proof}

\begin{thm}
The maximum number of cliques in a $K_8$-minor-free graph on $n\geq 6$ vertices equals $64(n-5)$. 
\end{thm}

\begin{proof}
\cref{Cockade} implies that a $(K_{2,2,2,2,2},5)$-cockade is far from extremal for the total number of cliques (since $\frac{211}{5}n - 179 < 64(n-5)$). The result then follows from Theorem~\ref{Jorgensen} and Theorem~\ref{K8}, with 6-trees providing the extremal example (see \eqref{LowerBound}).
\end{proof}

\section{\boldmath $K_9$-minor-free graphs}
\label{sectK9}

\citet{ST06} determined the extremal function for $K_9$-minors and characterised the extremal examples. 

\begin{thm}[\citet{ST06}]
\label{SongThomas}
Every $K_9$-minor-free graph on $n\geq 7$ vertices has at most $7n-28$ edges or is a $(K_{1,2,2,2,2,2},6)$-cockade or is isomorphic to $K_{2,2,2,3,3}$.
\end{thm}

Essentially the same method used above determines the maximum number of $k$-cliques in a $K_9$-minor-free graph for $k\geq 4$. 

\begin{thm}
\label{K9}
For  $k\in\{4,5,6,7,8\}$, the maximum number of $k$-cliques in a $K_9$-minor-free graph on $n\geq 7$ vertices equals
$$\binom{7}{k-1}\left(n-\frac{8(k-1)}{k}\right).$$
\end{thm}

\begin{proof}
The lower bound is provided by \eqref{LowerBound}. We proceed by induction on $n$ with $k$ fixed. In the base case with $n\in\{7,8\}$, the result holds by \eqref{smalln}. Let $G$ be a $K_9$-minor-free graph on $n\geq 9$ vertices. 

Say $G$ is a $(K_{1,2,2,2,2,2},6)$-cockade. 
Since $\cliques(K_{1,2,2,2,2,2},4)=1\cdot\binom{5}{3}\cdot2^3+\binom{5}{4}2^4=160$, by \cref{Pasting} with $r=6$ and $k=4$ 
we have $\cliques(G,4)=29(n-6)+\binom{6}{4}< 35(n-6)=\binom{7}{k-1}(n-\frac{8(k-1)}{k})$. 
Since $\cliques(K_{1,2,2,2,2,2},5)=1\cdot\binom{5}{4}\cdot 2^4+2^5=112$, by \cref{Pasting} with $r=6$ and $k=5$ we have 
$\cliques(G,5)=\frac{106}{5}(n-6)+6< 35(n-\frac{32}{5})= \binom{7}{k-1}(n-\frac{8(k-1)}{k})$. 
Since $\cliques(K_{1,2,2,2,2,2},6)=2^5=32$, by \cref{Pasting} with $r=k=6$  we have 
$\cliques(G,6)= \frac{31}{5}(n-6)+1\leq 21(n-\frac{20}{3})= \binom{7}{k-1}(n-\frac{8(k-1)}{k})
$. For $k\in\{7,8\}$ we have $\cliques(G,k)=0$.

If $G\cong K_{2,2,2,3,3}$, then $\cliques(G,4)=3\cdot2^2\cdot 3^2+2\cdot2^3\cdot 3=156<210=\binom{7}{3}(12-\frac{8\cdot 3}{4})$ 
and $\cliques(G,5)=2^3\cdot3^2< 196=\binom{7}{4}(12-\frac{8\cdot 4}{5})$. For $k\in\{6,7,8\}$ we have $\cliques(G,k)=0$.

We may now assume that $G$ is not a $(K_{1,2,2,2,2,2},6)$-cockade and  $G\not\cong K_{2,2,2,3,3}$. By \cref{SongThomas}, $|E(G)|\leq 7n-28$. 

The remainder of the proof is analogous to the proof of Theorem~\ref{MaderGen}, so we sketch it briefly. 
First, delete a vertex of degree at most 6 and apply induction. Now assume minimum degree at least 7. 
Charge each $k$-clique to its vertices, and count the charge at each vertex $v$ with respect to $\deg(v)$ and the number of $(k-1)$-cliques in $G(v)$, which is $K_8$-minor-free (applying Theorem~\ref{K8} since $k-1\geq 3$). 
Counting the total charge, \eqref{CliquesEdges} gives
$$\cliques(G,k) \leq \frac{1}{k}\binom{6}{k-2}\left( 2|E(G)|-\frac{(k-2)(7)n}{k-1}\right)$$
Since $|E(G)|\leq 7n-28$, it follows by the same analysis used in 
the proof of Theorem~\ref{MaderGen} that  $\cliques(G,k)\leq\binom{7}{k-1}(n-\frac{8(k-1)}{k})$. 
\end{proof}

With a bit more work, we now determine the maximum number of triangles in a $K_9$-minor-free graph. 
%

\begin{thm}
\label{K9triangles}
Every $K_9$-minor-free graph on $n\geq 7$ vertices contains at most $21n-112$ triangles, except for $K_{1,2,2,2,2,2}$ which has $120=21n-111$ triangles. 
\end{thm}

\begin{proof}
We proceed by induction on $n\geq 7$. Let $G$ be a $K_9$-minor-free graph with $n$ vertices. Assume the result holds for such graphs with less than $n$ vertices. If $n=7$ then $G$ contains at most $\binom{7}{3}=35=21 \cdot 7 - 112$ triangles. Now assume that $n\geq 8$. 

Case 1. $G\cong K_{2,2,2,3,3}$: Then $n=12$ and $G$ has $134< 21 n - 112$ triangles. Now assume that $G\not\cong K_{2,2,2,3,3}$

Case 2. $G\cong K_{1,2,2,2,2,2}$: Then $G$ contains $\binom{5}{3}2^3=80$ triangles that avoid the dominant vertex, and $40$ triangles that include the dominant vertex.  Thus $G$ contains $120=21\cdot 11-111$ triangles, as claimed. Now assume that $G\not\cong K_{1,2,2,2,2,2}$.  

Case 3. $G$ is a $(K_{1,2,2,2,2,2},6)$-cockade: Then $n>11$ as otherwise Case 2 applies. Using the calculation in Case 2, it follows from \cref{Pasting} that $G$ contains $20n-100$ triangles, which is less than $21n-111$ since $n>11$. Now assume that 
$G$ is not a $(K_{1,2,2,2,2,2},6)$-cockade.

Case 4. $\deg(v)\leq 7$ for some vertex $v$ in $G$: 
First suppose that $G-v\not \cong K_{1,2,2,2,2,2}$. 
By induction, $G-v$ contains at most $21(n-1)-112$ triangles. 
The number of triangles that include $v$ equals $|E(G(v))|$, which is at most $\binom{7}{2}= 21$.  
In total, $G$ contains at most $21n-112$ triangles. 
Now assume that $G-v\cong K_{1,2,2,2,2,2}$, which contains 120 triangles. 
In this case, $|E(G(v))|\leq \binom{7}{2}-1= 20$ since $K_{1,2,2,2,2,2}$ contains no $K_7$-subgraph.
In total, $G$ contains at most $120+20= 21 n-112$ triangles (since $n=12$), as claimed. 
Now assume that $G$ has minimum degree at least 8. 

Case 5. $G$ has a separation $(G_1,G_2)$ of order at most 5: That is, $G=G_1\cup G_2$ and $G\neq G_1$ and $G\neq G_2$ and $|V(G_1\cap G_2)|\leq 5$. 
Let $n_i:=|V(G_i)|$. Since each vertex in $G_1-V(G_2)$ has degree at least 8, $n_1\geq 9$. Similarly $n_2\geq 9$. By induction, $\cliques(G_i,3)\leq 21n_i -111$. Since $\cliques(G,3)\leq \cliques(G_1,3)+\cliques(G_2,3)$,
\begin{align*}
\cliques(G,3)\leq 21(n_1+n_2)-222 &\leq 21(n+5) -222 = 21n - 117 < 21 n - 112,
\end{align*} as desired.  Now assume that $G$ is 6-connected. 

Say $G$ has $m$ edges. By \cref{K9trianglesEdges} below, $G$ contains at most $4m-7n$ triangles, which is at most $4(7n-28)-7n=21n-112$ by \cref{SongThomas} (which is applicable since  $G$ is neither a $(K_{1,2,2,2,2,2},6)$-cockade nor isomorphic to $K_{2,2,2,3,3}$).
\end{proof}

\begin{lem}
\label{K9trianglesEdges}
Every 6-connected $n$-vertex $m$-edge $K_9$-minor-free graph contains at most $4m-7n$ triangles.
\end{lem}

\begin{proof}
First suppose $G$ contains a vertex $v$ with $G(v)$ isomorphic to a $(K_{2,2,2,2,2},5)$-cockade. 

Say $G-N[v]$ is empty. By \cref{Cockade}, $G(v)$ has $6n-26$ edges and $14n-74$ triangles. Thus $G$ has $m=7n-27$ edges and $20n-100$ triangles. Since $20n-100 \leq 4(7n-27) - 7n$, we are done. Now assume that $G-N[v]$ is not empty. Let $C$ be a connected component of $G-N[v]$. 
Then $N(C)\subseteq N(v)$. 

If $N(C)$ is a clique, then $|N(C)|\leq 5$ (since no $(K_{2,2,2,2,2},5)$-cockade contains $K_6$) and $G$ is not 6-connected, which is a contradiction. 
Thus $N(C)$ contains two non-adjacent vertices $x$ and $y$. 
Let $G'$ be obtained from $G$ by contracting $C$ to a vertex $z$ and then contracting $zx$. 
Then $xy$ is an edge of $G'$. Since every $(K_{2,2,2,2,2},5)$-cockade is edge-maximal with no $K_8$-minor, 
$G'[N(v)]$ contains a $K_8$-minor, and (with $v$) $G'$ contains a $K_9$-minor. Hence $G$ contains a $K_9$-minor, which is a contradiction. 

Now assume that $G(v)$ is isomorphic to a $(K_{2,2,2,2,2},5)$-cockade for no vertex $v$. Send a charge of $\frac{1}{3}$ from each triangle to each of the three vertices in it. Each vertex $v$ receives a charge equal to $\frac{1}{3}|E(G(v))|$, which is at most $2\deg(v)-7$ by
 \cref{Jorgensen} (which is applicable since $\deg(v)\geq 6$ and $G(v)$ is $K_8$-minor-free). The number of triangles, which equals the total charge, is at most $\sum_v(2\deg(v)-7)=4m-7n$, as desired. 
\end{proof}

\begin{thm}
The maximum number of cliques in a  $K_9$-minor-free graph on $n\geq 7$ vertices equals $128(n-6)$. 
\end{thm}

\begin{proof}
Let $G$ be a $K_9$-minor-free graph on $n\geq 7$ vertices. 

First suppose that $G$ is isomorphic to a $(K_{1,2,2,2,2,2},6)$-cockade. 
Note that $\cliques(K_{1,2,2,2,2,2})=2\cdot 3^5=486=\frac{422}{5}(11-6)+2^6$. By \cref{TotalPasting} with $a=\frac{422}{5}$ and $r=6$, every $n$-vertex 
$(K_{1,2,2,2,2,2},6)$-cockade contains $\frac{422}{5}(n-6)+64\leq 128(n-6)$ cliques. 
Now assume that $G$ is isomorphic to no $(K_{1,2,2,2,2,2},6)$-cockade.

Now suppose that $G\cong K_{2,2,2,3,3}$. Then $\cliques(G)=3^3\cdot 4^2=432<128\cdot 6=128(n-6)$. Now assume that $G\not\cong K_{2,2,2,3,3}$.

By \cref{SongThomas} and \cref{K9} and \cref{K9triangles}, we have
$\cliques(G,k)\leq \binom{7}{k-1}(n-\frac{8(k-1)}{k})$ for $k\in\{1,2,\dots,8\}$. 
Since $\cliques(G,0)=1$, we have  $\cliques(G)\leq 128(n-6)$.
\end{proof}

\section{Total Number of Cliques in $K_t$-minor-free Graphs}
\label{KtTotal}

This section considers the total number of cliques in $K_t$-minor-free graphs for arbitrary $t$. Recall that $\cliques(n,t)$ is the maximum number of cliques in a $K_t$-minor-free graph on $n$ vertices. 
The best lower bound on $\cliques(n,t)$ is due to \citet{Wood-GC07}, who observed that $K_{c\times 2}$ contains no $K_t$-minor where $t=\floor{\tfrac32c}+1$. Thus 
\begin{align}
\label{Kc2}
\cliques(n,t)\geq \cliques(K_{c\times 2})=3^c=2^{2(\log_23)t/3-o(t)}n\geq 2^{1.0566t-o(t)}n.
\end{align}
Upper bounds on $\cliques(n,t)$ have been intensely studied over the past ten years, culminating in the recent upper bound by \citet{FW16} that matches the lower bound  in \eqref{Kc2} up to a lower order term. These results are summarised in the following table. 
\begin{center}
\begin{tabular}{ll}
\hline
$\cliques(n,t)  \leq$ & reference\\
\hline
$(ct\sqrt{\log t})^tn$ & \citet{NSTW-JCTB06}\\
$2^{ct\sqrt{\log t}}n$ & \citet{ReedWood-TALG}\\
$2^{ct\log\log t}n$& \citet{FOT}\\
$2^{50t}n$ & \citet{LO15} \\
$2^{5t+o(t)}n$ & \citet{LO15}\\
$2^{2(\log_23)t/3+o(t)}n$ & \citet{FW16}\\
\hline
\end{tabular}
\end{center}

Note that several authors have also studied the maximum number of cliques in graphs excluding a given subdivision \citep{LO15,FW16a} or immersion \citep{FW16a}. 


The remainder of this section considers the following question: what is the maximum integer $t_0$ such that $\cliques(n,t)=2^{t-2}(n-t+3)$ for all $t\leq t_0$ (thus matching the lower bound in \eqref{TotalLowerBound})?  \cref{SummaryTotal} shows that $t_0 \geq 9$.

Given that $K_{c\times 2}$ provides an essentially tight lower bound on $\cliques(n,t)$, we now examine complete multipartite graphs in more detail. Consider a complete $c$-partite graph $G=K_{n_1,\dots,n_c}$ where $n_1\geq\dots\geq n_c\geq 1$ and $c\geq 2$. Then $n=\sum_{i=1}^cn_i$ is the number of vertices. \citet{Wood-GC07} proved that $G$ contains no $K_t$-minor, where 
$$t=\min\left\{\left\lfloor{\frac12 (n+c)}\right\rfloor+1,n-n_1+2\right\},$$ and a $K_{t-1}$-minor in $G$ can be obtained from a $K_c$ subgraph by contracting a maximum matching in the remaining graph. If $\floor{\frac12 (n+c)}\leq n-n_1+1$ then we say $G$ is \emph{balanced}, otherwise $G$ is \emph{unbalanced} (in which case the largest colour class is `very' large). First suppose that $G$ is unbalanced. Let $m:=\frac{n-n_1}{c-1}$ be the average size of a colour class except the largest colour class. Then $m\geq 1$ and $m+1\leq 2^m$. Thus 
$$\cliques(G)=\prod_{i=1}^c(n_i+1) 
\leq (n_1+1)(m+1)^{c-1}
\leq (n_1+1)2^{m(c-1)}
= (n-t+3)2^{t-2}.$$
That is, every unbalanced complete multipartite graph satisfies the bound. Now consider the case in which $G$ is balanced. Let $m:=\frac{n}{c}$ be the average size of a colour class. 
Then $t-2\geq \frac12(n+c-1)-1=\frac12(c(m+1)-3)$ and $n-t+3\geq n-\frac12 (n+c)+3 =\frac12(n-c)+3=\frac12 c(m-1)+3.$ 
Assume that $m\geq 3$. Then $(m+1)^{2}\leq 2^{m+1}$ and $(m+1)^{c}\leq 2^{c(m+1)/2}$. Since $\half c(m-1)+3 > 2^{3/2}$, 
$$2^{3/2}(m+1)^c\leq 2^{c(m+1)/2}( \half c(m-1)+3)  \leq 2^{c(m+1)/2}( n-t+3)$$
and
$$\cliques(G)=\prod_{i=1}^c(n_i+1) \leq (m+1)^c \leq 2^{(c(m+1)-3)/2}( n-t+3)\leq 2^{t-2}( n-t+3).$$
That is, balanced complete multipartite graphs with an average of at least three vertices per colour class satisfy the bound. Thus, if a complete multipartite graph has more than $2^{t-2}(n-t+3)$ cliques, then it is balanced and has an average of less than three vertices per colour class. This is why $K_{c\times 2}$ is a critical example. Computer search establishes that for $t\leq 49$ every such complete multipartite graph has at most $2^{t-2}(n-t+3)$ cliques, but for $t\geq 50$ there is a value of $c$ such that $K_{c\times 2}$ or $K_{1,c\times 2}$ or $K_{1,1,c\times 2}$ has no $K_t$-minor and contains more than $2^{t-2}(n-t+3)$ cliques. Indeed $K_{c\times 2}$ satisfies this property for $t\geq 62$. We therefore make the following conjecture.

\begin{conj}
$\cliques(n,t)=2^{t-2}(n-t+3)$ if and only if  $t\leq 49$. 
\end{conj}

%
%

\section{Number of $k$-Cliques in $K_t$-minor-free Graphs}
\label{KtSpecific}

This section considers the maximum number of $k$-cliques in $K_t$-minor-free graphs. First note that \cref{Mader} fails badly for large $t$. In particular,  \citet{Kostochka82, Kostochka84} and de la Vega~\cite{delaVega} (based on the work of \citet{BCE80}) proved that $\cliques(n,t,2)\geq c_1t\sqrt{\log t}n$ for some constant $c_1>0$. Conversely, \citet{Kostochka82, Kostochka84} and \citet{Thomason84} proved that $\cliques(n,t,2)\leq c_2t\sqrt{\log t}n$  for some constant $c_2>0$. Later, \citet{Thomason01} proved that $\cliques(n,t,2)=(\alpha +o(1))nt\sqrt{\ln t}$, where $\alpha\approx 0.319$ is precisely determined.

A graph is \emph{$d$-degenerate} if every subgraph has minimum degree at most $d$. \citet{Wood-GC07} determined the maximum total number of cliques in a $d$-degenerate graph. Essentially the same proof determines the maximum number of $k$-cliques. 

\begin{lem}
\label{degen}
For every $d$-degenerate graph $G$ with $n\geq d+1$ vertices, $$\cliques(G,k)\leq \binom{d}{k-1}\left(n-\frac{(k-1)(d+1)}{k}\right).$$
\end{lem}

\begin{proof} 
%
We proceed by induction on $n$. For the base case with $n=d+1$, the number of $k$-cliques is at most $\binom{n}{k}=\binom{d}{k-1}(n-\frac{k-1}{k}(d+1))$. Let $G$ be a $d$-degenerate graph with $n\geq d+2$ vertices. There is a vertex $v$ of degree at most $d$ in $G$. The number of $k$-cliques containing $v$ is at most $\binom{d}{k-1}$. The number of $k$-cliques not containing $v$ (that is, in $G-v$) is at most $\binom{d}{k-1}(n-1-\frac{k-1}{k}(d+1))$ by induction (since $G-v$ is also $d$-degenerate). In total, the number of $k$-cliques is at most $\binom{d}{k-1}(n-\frac{k-1}{k}(d+1))$ .
\end{proof}

The bound in \cref{degen} is tight for $d$-trees. The above-mentioned results of \citet{Kostochka82,Kostochka84} and \citet{Thomason84,Thomason01} show that $K_t$-minor-free graphs are $ct\sqrt{\log t}$-degenerate for some constant $c$. \cref{degen} thus implies
$$\cliques(t,n,k) \leq \binom{ct\sqrt{\log t}}{k-1}n \leq (ct\sqrt{\log t})^{k-1}n.$$ 
For fixed $k$, this bound is tight up to a constant factor  as we now explain. \citet{BCE80} proved that for a suitable constant $c>0$ and for large $t$, a random graph  on $n=ct\sqrt{\log t}$ vertices has no $K_t$-minor with high probability. Here each edge is chosen independently with probability $\frac12$. Thus, the expected number of $k$-cliques is $\binom{n}{k}/2^{\binom{k}{2}}$. It follows that for large $t$, there exists an $n$-vertex graph with no $K_t$-minor and with at least $\binom{n}{k}/2^{\binom{k}{2}}$ $k$-cliques. Note that $\binom{n}{k}/2^{\binom{k}{2}}=c'(t\sqrt{\log t})^{k-1}n$ for a suitable constant $c'$. Thus there exists an $n$-vertex graph $G$ with no $K_t$-minor such that $\cliques(G,k)\geq c'(t\sqrt{\log t})^{k-1}n$. Taking disjoint copies of $G$ gives a graph with the same property, where $n\gg t$.  Summarising, for fixed $k$, there are constants $c_1$ and $c_2$ such that 
\begin{equation*}
c_1(t\sqrt{\log t})^{k-1}n \leq \cliques(n,t,k)\leq c_2 (t\sqrt{\log t})^{k-1}n.
\end{equation*}
Thus $\cliques(n,t,k)$ is determined up to a constant factor for fixed $k$. But as $k$ increases with $t$, determining $\cliques(n,t,k)$ is wide open. 
First note that a random graph will have few large cliques.  In fact, the size of the largest clique in a random graph on $t$ vertices is sharply concentrated around $2\log_2 t$ \cite{Matula70,BolErd76}. This motivates the following conjecture about `large' cliques in $K_t$-minor-free graphs.


\begin{conj}
\label{LogConj}
For some constants $c_1,c_2>0$, for all integers $t\geq 3$ and $k\geq c_1\log t$ and $n\geq t-1$, 
$$\cliques(n,t,k)\leq (c_2t)^kn.$$
\end{conj}

This conjecture is true for $k\geq \frac{2(\log_23)t/3+o(t)}{\log_2 t}$, since the above-mentioned upper bound of \citet{FW16} implies:  
\begin{align*}
\cliques(n,t,k) \leq \cliques(n,t)  \leq 2^{2(\log_23)t/3+o(t)}n \leq t^kn. 
\end{align*}

For very large cliques in $K_t$-minor-free graphs, we conjecture that the lower bound in \eqref{LowerBound} is tight. 

\begin{conj}
\label{MainConj}
For some $\lambda\in[\tfrac13,1)$, for all integers $t\geq 3$ and $k\geq \lambda t$ and $n\geq t-1$, 
$$\cliques(n,t,k)\leq \binom{t-2}{k-1}\left(n-\frac{(k-1)(t-1)}{k}\right).$$
\end{conj}

We now provide two pieces of evidence in support of this conjecture. First, even though $K_{c\times 2}$ contains many cliques in total (see \eqref{Kc2}), it contains no clique of order $c+1\approx \frac23t$, implying $K_{c\times 2}$ satisfies \cref{MainConj} for $\lambda=\frac{2}{3}$. Our second piece of evidence is to prove \cref{MainConj} for $k=t-1$ (which is the largest non-trivial value of $k$). 

\begin{prop}
For all integers $t\geq 2$ and $n\geq t-1$, 
$$\cliques(n,t,t-1)= n-t+2.$$
\end{prop}

\begin{proof}
The lower bound is provided by \eqref{LowerBound}. For the upper bound, we proceed by induction on $n$. Let $G$ be a $K_t$-minor-free graph on $n\geq t-1$ vertices. If $n=t-1$ then $\cliques(G,t-1)\leq 1$ as desired. Now assume that $n\geq t$. We may assume that $G$ contains $K_{t-1}$.  Let $K$ be the vertex set of a copy of $K_{t-1}$ in $G$. Let $X_1,\dots,X_r$ be the vertex sets of the connected components of $G-K$. If for some $i\in[1,r]$, every vertex in $K$ has a neighbour in $X_i$, then contracting $G[X_i]$ to a single vertex (with $K$) gives a $K_t$-minor in $G$. Thus, for each $i\in[1,r]$, some vertex $v_i$ in $K$ is adjacent to no vertex in $X_i$. Each copy of $K_{t-1}$ distinct from $K$ is contained in $G_i:=G[(X_i\cup K)\setminus\{v_i\}]$ for some $i\in[1,r]$. Thus $$\cliques(G,t-1)=1+ \sum_{i=1}^r\cliques(G_i,t-1).$$ Note that each $G_i$ has $|X_i|+t-2$ vertices, which is at least $t-1$ and less than $n$. By induction, 
\begin{equation*}
\cliques(G,t-1)\leq 1+ \sum_{i=1}^r|X_i| = 1+ n-(t-1)=n-t+2.\qedhere
\end{equation*}
\end{proof}

Finally we justify the lower bound of $\lambda\geq\tfrac13$ in \cref{MainConj}. Again, $K_{c\times 2}$ is the example. Assume $c$ is even. As noted earlier, $G$ contains no $K_t$-minor where $t=\tfrac32 c+1$. Observe that $\cliques(K_{c\times 2},k)=\binom{c}{k}2^k$. The following series of conditions are equivalent. 
\begin{align}
\cliques(K_{c\times 2},k) & \leq \binom{t-2}{k-1}\left(n - \frac{(k-1)(t-1)}{k}\right)\nonumber\\
\binom{c}{k}2^k & \leq \binom{t-2}{k-1}\left(2c - \left(1-\frac{1}{k}\right)\frac{3c}{2}\right)\nonumber\\
\frac{c!\,2^k}{(c-k)!k!} & \leq \frac{(t-2)!}{(k-1)!(t-k-1)!}\left(\frac{c}{2} + \frac{3c}{2k}\right)\nonumber\\
\frac{c!\,2^{k+1}}{(c-k)!} & \leq \frac{(t-2)!\,c(k+3)}{(t-k-1)!}\nonumber\\
2^{k+1}\prod_{i=1}^{k-1}(c-i) & \leq (k+3) \prod_{i=1}^{k-1}(t-i-1)\nonumber\\
2^{k+1}\prod_{i=1}^{k-1}(c-i) & \leq (k+3) \prod_{i=1}^{k-1}\left(\frac{3c}{2}-i\right)\label{K222iff}
\end{align}
Thus $K_{c\times 2}$ satisfies \cref{MainConj} for a particular value of $k$ if and only if \eqref{K222iff} holds. 
We now show that  \eqref{K222iff} is not satisfied by $K_{c\times 2}$ for small $k$; that is, $K_{c\times 2}$ has many small cliques. Fix $\epsilon>0$. 
Let $k\geq k(\epsilon)$ and $c\geq (2+\epsilon)k$. 
Since $1\leq i< k\leq\frac{c}{2+\epsilon}$, it follows that $2(c-i) \geq (1+\frac{\epsilon}{4+3\epsilon})(\tfrac{3c}{2}-i)$. Thus
$$2^{k+1}\prod_{i=1}^{k-1}(c-i) > 4\left(1+\frac{\epsilon}{4+3\epsilon}\right)^{k-1}\prod_{i=1}^{k-1}\left(\frac{3c}{2}-i\right).$$
Now $4(1+\frac{\epsilon}{4+3\epsilon})^{k-1}> k+3$ for large $k\geq k(\epsilon)$. Thus
 \eqref{K222iff} is not satisfied, and 
$$\cliques(K_{c\times 2},k) > \binom{t-2}{k-1}\left(n-\frac{(k-1)(t-1)}{k}\right).$$
Thus $k>\frac{c}{2+\epsilon}$ for \cref{MainConj} to hold for $K_{c\times 2}$. Since $c\approx\frac23t$, this says that $\lambda\geq\frac13-\epsilon$ in \cref{MainConj}.

\subsection*{Acknowledgements} Many thanks to both referees for their very thorough reviews.


\def\soft#1{\leavevmode\setbox0=\hbox{h}\dimen7=\ht0\advance \dimen7
  by-1ex\relax\if t#1\relax\rlap{\raise.6\dimen7
  \hbox{\kern.3ex\char'47}}#1\relax\else\if T#1\relax
  \rlap{\raise.5\dimen7\hbox{\kern1.3ex\char'47}}#1\relax \else\if
  d#1\relax\rlap{\raise.5\dimen7\hbox{\kern.9ex \char'47}}#1\relax\else\if
  D#1\relax\rlap{\raise.5\dimen7 \hbox{\kern1.4ex\char'47}}#1\relax\else\if
  l#1\relax \rlap{\raise.5\dimen7\hbox{\kern.4ex\char'47}}#1\relax \else\if
  L#1\relax\rlap{\raise.5\dimen7\hbox{\kern.7ex
  \char'47}}#1\relax\else\message{accent \string\soft \space #1 not
  defined!}#1\relax\fi\fi\fi\fi\fi\fi}

\end{document}